\documentclass[12pt]{amsart}

\usepackage{amsmath,amssymb,enumitem,verbatim,stmaryrd,xcolor,microtype,graphicx,mathscinet,mathtools}
\usepackage[T1]{fontenc}
\usepackage[utf8]{inputenc}
\usepackage[english]{babel}
\usepackage[top=3.2cm,bottom=3.2cm,left=3.2cm,right=3.2cm]{geometry}
\usepackage[bookmarksdepth=2,linktoc=page,colorlinks,linkcolor={red!80!black},citecolor={red!80!black},urlcolor={blue!80!black}]{hyperref}
\usepackage[all]{xy}
\usepackage[mathcal]{euscript}                               
\usepackage{mathptmx}                                        
\usepackage[backgroundcolor=yellow,linecolor=yellow,textsize=footnotesize]{todonotes} \makeatletter \providecommand \@dotsep{5} \def\listtodoname{List of Todos} \def\listoftodos{\@starttoc{tdo}\listtodoname} 
 \makeatother 

\usepackage{etoolbox}
\makeatletter
\patchcmd{\@startsection}{\@afterindenttrue}{\@afterindentfalse}{}{}             
\patchcmd{\part}{\bfseries}{\bfseries\LARGE}{}{}
\patchcmd{\section}{\scshape}{\bfseries}{}{}\renewcommand{\@secnumfont}{\bfseries} 
\patchcmd{\@settitle}{\uppercasenonmath\@title}{\large}{}{}
\patchcmd{\@setauthors}{\MakeUppercase}{}{}{}
\addto{\captionsenglish}{} 
\makeatother

\usepackage{fancyhdr}

\newtheorem{theorem}{Theorem}[section]
\newtheorem{lemma}[theorem]{Lemma}

\theoremstyle{definition}
\newtheorem{definition}[theorem]{Definition}

\theoremstyle{remark}
\newtheorem{remark}[theorem]{Remark}

\numberwithin{equation}{section}

\DeclareRobustCommand{\gobblefour}[5]{}    

\DeclareFontFamily{OT1}{pzc}{}                                
\DeclareFontShape{OT1}{pzc}{m}{it}{<-> s * [1.10] pzcmi7t}{}
\DeclareMathAlphabet{\mathpzc}{OT1}{pzc}{m}{it}
\DeclareSymbolFont{sfoperators}{OT1}{bch}{m}{n} \DeclareSymbolFontAlphabet{\mathsf}{sfoperators} \makeatletter\def\operator@font{\mathgroup\symsfoperators}\makeatother 
\DeclareSymbolFont{cmletters}{OML}{cmm}{m}{it}              
\DeclareSymbolFont{cmsymbols}{OMS}{cmsy}{m}{n}
\DeclareSymbolFont{cmlargesymbols}{OMX}{cmex}{m}{n}
\DeclareMathSymbol{\myjmath}{\mathord}{cmletters}{"7C}     \let\jmath\myjmath 
\DeclareMathSymbol{\myamalg}{\mathbin}{cmsymbols}{"71}     
\DeclareMathSymbol{\mycoprod}{\mathop}{cmlargesymbols}{"60}
\DeclareMathSymbol{\myalpha}{\mathord}{cmletters}{"0B}     \let\alpha\myalpha 
\DeclareMathSymbol{\mybeta}{\mathord}{cmletters}{"0C}      \let\beta\mybeta
\DeclareMathSymbol{\mygamma}{\mathord}{cmletters}{"0D}     \let\gamma\mygamma
\DeclareMathSymbol{\mydelta}{\mathord}{cmletters}{"0E}     \let\delta\mydelta
\DeclareMathSymbol{\myepsilon}{\mathord}{cmletters}{"0F}   \let\epsilon\myepsilon
\DeclareMathSymbol{\myzeta}{\mathord}{cmletters}{"10}      \let\zeta\myzeta
\DeclareMathSymbol{\myeta}{\mathord}{cmletters}{"11}       \let\eta\myeta
\DeclareMathSymbol{\mytheta}{\mathord}{cmletters}{"12}     \let\theta\mytheta
\DeclareMathSymbol{\myiota}{\mathord}{cmletters}{"13}      \let\iota\myiota
\DeclareMathSymbol{\mykappa}{\mathord}{cmletters}{"14}     \let\kappa\mykappa
\DeclareMathSymbol{\mylambda}{\mathord}{cmletters}{"15}    \let\lambda\mylambda
\DeclareMathSymbol{\mymu}{\mathord}{cmletters}{"16}        \let\mu\mymu
\DeclareMathSymbol{\mynu}{\mathord}{cmletters}{"17}        \let\nu\mynu
\DeclareMathSymbol{\myxi}{\mathord}{cmletters}{"18}        \let\xi\myxi
\DeclareMathSymbol{\mypi}{\mathord}{cmletters}{"19}        \let\pi\mypi
\DeclareMathSymbol{\myrho}{\mathord}{cmletters}{"1A}       \let\rho\myrho
\DeclareMathSymbol{\mysigma}{\mathord}{cmletters}{"1B}     \let\sigma\mysigma
\DeclareMathSymbol{\mytau}{\mathord}{cmletters}{"1C}       \let\tau\mytau
\DeclareMathSymbol{\myupsilon}{\mathord}{cmletters}{"1D}   \let\upsilon\myupsilon
\DeclareMathSymbol{\myphi}{\mathord}{cmletters}{"1E}       \let\phi\myphi
\DeclareMathSymbol{\mychi}{\mathord}{cmletters}{"1F}       \let\chi\mychi
\DeclareMathSymbol{\mypsi}{\mathord}{cmletters}{"20}       \let\psi\mypsi
\DeclareMathSymbol{\myomega}{\mathord}{cmletters}{"21}     \let\omega\myomega
\DeclareMathSymbol{\myvarepsilon}{\mathord}{cmletters}{"22}\let\varepsilon\myvarepsilon
\DeclareMathSymbol{\myvartheta}{\mathord}{cmletters}{"23}  \let\vartheta\myvartheta
\DeclareMathSymbol{\myvarpi}{\mathord}{cmletters}{"24}     \let\varpi\myvarpi
\DeclareMathSymbol{\myvarrho}{\mathord}{cmletters}{"25}    \let\varrho\myvarrho
\DeclareMathSymbol{\myvarsigma}{\mathord}{cmletters}{"26}  \let\varsigma\myvarsigma
\DeclareMathSymbol{\myvarphi}{\mathord}{cmletters}{"27}    \let\varphi\myvarphi

\newcommand\F{{\mathbb F}}

\renewcommand\H{{\mathbb H}}
\newcommand\K{{\mathbb K}}
\newcommand\N{{\mathbb N}}

\newcommand\Q{{\mathbb Q}}
\newcommand\R{{\mathbb R}}
\renewcommand\S{{\mathbb S}}

\newcommand\W{{\mathbb W}}

\newcommand\cP{{\mathcal P}}

\newcommand\ord{\textup{ord}}

\renewcommand\geq{\geqslant}
\renewcommand\leq{\leqslant}

\newcommand{\hyperplus}{\mathrel{\,\raisebox{-1.1pt}{\larger[-0]{$\boxplus$}}\,}}

\newdir{ >}{{}*!/-5pt/@^{(}}\newcommand{\arincl}[1]{\ar@{ >->}@<-0,0ex>#1} 

\title{On the Structure of Hyperfields Obtained as Quotients of Fields}

\author{Matthew Baker}
\email{mbaker@math.gatech.edu}
\address{School of Mathematics, Georgia Institute of Technology, Atlanta, USA}

\author{Tong Jin}
\email{tjinmath@outlook.com}
\address{Taishan College, Shandong University, Jinan, China}

\begin{document}


\begin{abstract} 
We determine all isomorphism classes of hyperfields of a given finite order which can be obtained as quotients of finite fields of sufficiently large order. Using this result, we determine which hyperfields of order at most 4 are quotients of fields. The main ingredients in the proof are the Weil bounds from number theory and a result from Ramsey theory.
\end{abstract}

\thanks{{\bf Acknowledgements.} We thank Trevor Gunn for catching some typos in the first draft of the paper, Keith Conrad for historical remarks concerning Theorem~\ref{thm:Weil}, and the anonymous referee for his/her careful reading and useful suggestions.}

\maketitle

\section{Introduction}
\label{intro}

If $F$ is a field and $G$ is a multiplicative subgroup of $F^\times$, the quotient $F/G$ naturally has the structure of a {\em hyperfield} in the sense of M.~Krasner; 
 in particular, addition is a {\em multi-valued} binary operation. Krasner asked in \cite{Krasner} if every abstract hyperfield arises from this quotient construction; it turns out that the answer is {\em no} (the first counterexample was found by Massouros \cite{Massouros}). Nevertheless, it is still of interest to classify quotient hyperfields, and the simplest case is that of quotients of {\em finite fields}. In this case, for each $r$ dividing $q-1$ there is a unique subgroup $G^r_q$ of ${\mathbb F}_q^\times$ of index $r$, and the quotient ${\mathbb F}_q / G^r_q$ is a hyperfield of order $r+1$. 

\medskip

It is natural to fix $r$ and vary $q$ and ask how many different hyperfields of order $r+1$ one obtains from this construction. For $r=1$, one always obtains either the so-called ``Krasner hyperfield'' ${\mathbb K}$ or the finite field $\F_2$, and for $r=2$ it is well-known that when $q\geq 7$ there are just 2 possibilities, depending on whether $q$ is congruent to 1 or 3 modulo 4. (For $q=3,5$ one obtains two additional ``sporadic'' quotient hyperfields of order $3$.) Our first main goal in this paper is to generalize these observations to arbitrary natural numbers $r$ (when the prime-power $q$ is sufficiently large). Using the Weil bounds for the number of points on algebraic curves over finite fields, we prove the following:

\begin{theorem}\label{first main theorem}
Given an integer $r\geq2$, there is a bound $N_r$ (which we can take to be $r^4$) such that the following holds: 
\begin{enumerate}
\item If $r$ is odd, there is a unique hyperfield $\mathbb{H}_r$ such that $\F_q/G^r_q$ is isomorphic to $\mathbb{H}_r$ for every prime power $q\geq N_r$ with $q\equiv 1 \pmod{r}$. 
\item If $r$ is even, there are two distinct hyperfields $\mathbb{H}_r$ and $\mathbb{H}'_r$ such that for every prime power $q\geq N_r$ with $q\equiv 1 \pmod{2r}$ (resp. $q\equiv r+1 \pmod{2r}$), $\F_q/G^r_q$ is isomorphic to $\mathbb{H}_r$ (resp. $\mathbb{H}'_r$). 
\end{enumerate}
\end{theorem}

\begin{remark} \label{rem:betterNr}
The proof will show, more precisely, that we can take $N_r$ to be the smallest positive integer $N$ such that $(N+1) -  (r-1)(r-2)N^{1/2} - 3r > 0$.
When $r=2$, this gives $N_2 = 6$ and when $r=3$, it gives $N_3 = 17$.
\end{remark}

We use this result to determine, among all isomorphism classes of hyperfields of order at most $4$, precisely which ones are quotients of fields (finite or infinite).
In addition to Theorem~\ref{first main theorem} and some routine case analysis, this requires a Ramsey-theoretic result proved independently by Bergelson--Shapiro and Turnwald \cite{BergelsonShapiro,Turnwald}.  
Our classification can be summarized as follows:

\begin{theorem}\label{second main theorem}
\begin{enumerate}
\item There are 2 isomorphism classes of hyperfields of order 2, the Krasner hyperfield and the finite field $\F_2$. Both of them are quotients of finite fields. 
\item Every hyperfield of order 3 is a quotient of a field. More precisely, there are 5 isomorphism classes of hyperfields of order 3, all but one of which are isomorphic to quotients of finite fields. The hyperfield of signs is isomorphic to $\R / \R_{>0}$ but not to any quotient of a finite field.
\item There are 7 isomorphism classes of hyperfields of order 4. Of these, 4 are quotients of finite fields, 1 is a quotient of an infinite field but not of any finite field, and the remaining 2 are not quotients of any field.
\end{enumerate}
\end{theorem}

\subsection*{Content overview}
In Section \ref{section: hyperfields}, we provide some background information on hyperfields. Theorem~\ref{first main theorem} will be proved in Section \ref{section: proof of the first main theorem} and Theorem~\ref{second main theorem} will be proved in Section~\ref{section: proof of the second main theorem}.  We conclude in Section \ref{section:open questions} with some open questions.


\section{Hyperfields}
\label{section: hyperfields}

\begin{definition}\label{def: hyperoperation}
A \emph{hyperoperation} on a set $F$ is a function
\[
 \hyperplus: \quad F\times F \quad \longrightarrow \quad \cP(F)
\]
from $F \times F$ to the power set of $F$ such that $a \hyperplus b$ is non-empty for all $a,b \in F$.
\end{definition}

A hyperoperation is called \emph{commutative} if $a \hyperplus b = b \hyperplus a$ for all $a,b$ $\in F$, and \emph{associative} if
\[
 \bigcup_{d\in b\hyperplus c} a\hyperplus d = \bigcup_{d\in a\hyperplus b} d\hyperplus c
\]
for all $a,b,c \in F$. 

\begin{definition}\label{def: hypergroup}
A \emph{hypergroup} is a set $G$ equipped with an associative hyperoperation $\hyperplus$ satisfying the following axioms: 
\begin{enumerate}
 \item\label{HG1} There is an element $0 \in G$ such that $0\hyperplus a=a\hyperplus 0=\{a\}$.  \hfill\textit{(identity)}
 \item\label{HG2} There is a unique element $-a$ in $G$ such that $0\in a\hyperplus (-a)$.  \hfill\textit{(inverses)}
 \item\label{HG3} $a\in b\hyperplus c$ if and only if $-b \in (-a)\hyperplus c$.    \hfill\textit{(reversibility)}
\end{enumerate}
\end{definition}

A hypergroup is called \emph{commutative} if the hyperoperation $\hyperplus$ is commutative.

\begin{definition}\label{def: hyperfield}
A \emph{hyperfield} is a set $\H$ equipped with a binary operation $\cdot$, a hyperoperation $\hyperplus$, and distinct elements $0,1 \in \H$ satisfying the following axioms: 
 \begin{enumerate}
 \item\label{HF1} $(\H,\hyperplus,0)$ is a commutative hypergroup.
 \item\label{HF2} $(\H^\times,\cdot,1)$ is an abelian group.
 \item\label{HF3} $a\cdot 0=0\cdot a = 0$ for all $a \in \H$.
 \item\label{HF4} $a\cdot(b\hyperplus c)=ab\hyperplus ac$ for all $a,b,c \in \H$, where $a\cdot(b\hyperplus c)$ is defined as $\{ad \; | \; d\in b\hyperplus c\}$.   \hfill\textit{(distributivity)}
\end{enumerate}
\end{definition}

It follows easily from the definitions that $(-1)^2 = 1$ in any hyperfield $\H$.

Some important examples of hyperfields are as follows: 
\begin{enumerate}
 \item Every field $F$ is tautologically a hyperfield by defining $a\hyperplus b=\{a+b\}$. 
 \item The \emph{Krasner hyperfield} $\K = \{0, 1\}$ is equipped with the usual multiplication and hyperaddition characterized by $1 \hyperplus 1 = \{0,1\}$. 
 \item The \emph{hyperfield of signs} $\S = \{0,1,-1\}$ is equipped with the usual multiplication and hyperaddition characterized by the rules $1\hyperplus 1=\{1\}$, $-1\hyperplus -1=\{-1\}$ and $1\hyperplus-1=\{0,1,-1\}$. 
 \item The \emph{weak hyperfield of signs} $ \W = \{0, 1, -1\}$ is equipped with the usual multiplication and hyperaddition characterized by the rules $1\hyperplus 1=-1\hyperplus -1=\{1, -1\}$ and $1\hyperplus-1=\{0,1,-1\}$. 
\end{enumerate}

Let $F$ be a field and let $G$ be a subgroup of $F^\times$. The multiplicative monoid $F/G = (F^\times / G) \cup \{ 0 \}$ can be endowed with a natural hyperfield structure by setting
\[
   [a] \hyperplus [b] \ = \ \big\{ \ [c]\ \big| \ c=a'+b'\text{ for some }a'\in [a], b'\in [b] \ \big\}.
\]

We call hyperfields of this form {\em quotient hyperfields}. The four examples given above are all quotient hyperfields: indeed, we have $\K = F/F^\times$ for any field $F$ with more than two elements, $\S = \R / \R_{>0}$, and $\W = \F_p / (\F_p^\times)^2$ for any prime number $p\geq 7$ with $p \equiv 3 \pmod{4}$.

\begin{definition}
Let $\H_1$ and $\H_2$ be hyperfields. A map $\phi: \H_1 \to \H_2$ is called a \emph{hyperfield homomorphism} if $\phi(0) = 0$, $\phi(1)=1$, $\phi(ab)=\phi(a)\phi(b)$, and $\phi(a \hyperplus_1 b) \subset \phi(a) \hyperplus_2 \phi(b)$ for all $a, b \in \H_1$.
\end{definition}

A homomorphism of hyperfields $\phi : \H_1 \to \H_2$ is an \emph{isomorphism} if there is an inverse homomorphism $\psi : \H_2 \to \H_1$. 


\section{Proof of Theorem~\ref{first main theorem}}
\label{section: proof of the first main theorem}

Our proof of Theorem~\ref{first main theorem} is based on the following well-known inequality from number theory:

\begin{theorem}[Davenport--Hasse] \label{thm:Weil}
Let $a$, $b$, and $c$ be nonzero elements of the finite field $\F_q$, and let $r$ be a positive integer with $r\mid q-1$. Then the number $M(q)$ of projective solutions in $\mathbb{P}^2(\F_q)$ to the homogeneous equation $ax^r+by^r+cz^r=0$ satisfies
$$\mid M(q) - (q+1)\mid \leq (r-1)(r-2)q^{1/2}.$$
\end{theorem} 

\begin{proof}
See, e.g., \cite[Corollary 2.5.23]{Cohen} for an ``elementary'' proof using Jacobi sums derived from the work of Davenport and Hasse \cite{DavenportHasse}. The Davenport--Hasse theorem is a special case of the more general result, proved by Weil, that if $C / \F_q$ is a nonsingular projective curve of genus $g$ and $C(\F_q)$ denotes the set of points of $C$ over $\F_q$ then $| \# C(\F_q) - (q+1) | \leq 2g q^{1/2}$.
\end{proof}

We now give the proof of Theorem~\ref{first main theorem}.

\begin{proof}[Proof of Theorem~\ref{first main theorem}] 

Let $\F_q$ be a finite field with $q=p^n$ elements, and let $g$ be a generator of the multiplicative group $\F_q^\times$. Then for $r \mid q-1$, the unique subgroup $G^r_q$ of $\F_q^\times$ of index $r$ is of the form 
$$\{g^r, g^{2r}, ..., g^{sr}\},$$
where $q=sr+1$. 

Given $i, j, k\in\{1, 2, ..., r\}$, the following are equivalent:
\begin{enumerate}
 \item\label{solu1} $[g^k] \in [g^i] \hyperplus [g^j]$ in $\H$. 
 \item\label{solu2} There exists at least one solution $(x,y,z) \in (\F_q^\times)^3$ to the homogeneous polynomial equation 
\begin{equation}\label{equa}
  g^ix^r+g^jy^r-g^kz^r=0.
\end{equation}
\end{enumerate}

For projective solutions to (\ref{equa}) with $x=0$, we have
$$g^jy^r-g^kz^r=0,$$
or equivalently, 
$$(z/y)^r = g^{j-k}.$$
Therefore the number of such solutions is at most $r$, and thus the equation (\ref{equa}) has at most $3r$ projective solutions with some coordinate equal to zero. 

For $q \geq r^4$, Theorem~\ref{thm:Weil} implies that the number $N(q)$ of projective solutions in $\F_q$ to
  $$g^ix^r+g^jy^r-g^kz^r=0$$ 
with $x$, $y$, and $z$ all nonzero satisfies
\[
\begin{split}
N(q) & \geq M(q) - 3r\\
 & \geq (q+1) -  (r-1)(r-2)q^{1/2} - 3r \\
 & = \sqrt{q} (\sqrt{q} - (r-1)(r-2)) -3r +1\\
 & \geq r^2(3r-2) -3r +1\\
 & >0
\end{split}
\]
for $r\geq 2$. 

It follows that $[g^k] \in [g^i] \hyperplus [g^j]$ for all $i, j, k \in \{ 1, 2,\ldots, r\}$, i.e., $\H^\times\subseteq  [g^i] \hyperplus [g^j]$ for every $i, j \in \{ 1, 2,\ldots, r\}$. 

It remains to determine for which $i$ we have $0 \in [1] + [g^i]$, that is, $[g^i] = [-1]$. Since $[g^i] = [g^j]$ if and only if $i$ and $j$ are congruent modulo $r$ and $[-1] = [g^{(q-1)/2}]$, this happens precisely when $i$ is congruent to $(q-1)/2$ modulo $r$. So it suffices to calculate the value of $(q-1)/2$ modulo $r$. 

\medskip\noindent\textbf{Case 1:} {$r$ is odd. }

\smallskip

We claim that $[-1] = [g^{(q-1)/2}]= [g^0]= [1]$, i.e.,  $0 \in [1]\hyperplus[1]$.

\medskip\noindent\textbf{Case 1a:} $p=2$. 
In this case, the inclusion $0 \in [1]\hyperplus[1]$ follows easily from the fact that
$x + x = 0$ for each $x \in \F_q$. 

\medskip\noindent\textbf{Case 1b:} $q = p^n$ is an odd prime power.
In this case, since both $q$ and $r$ are odd, we write $q = 2lr+1$. Then
\[
  \frac{q-1}{2} \equiv lr \equiv 0.
\]

It follows that $\H$ is isomorphic to the hyperfield $\mathbb{H}_r$ whose multiplicative group is cyclic of order $r$ and whose hyperaddition is characterized by the identities $-x = x$ for all $x \in \mathbb{H}_r$; $x \hyperplus x= \mathbb{H}_r$ for $x \ne 0$; and $x \hyperplus y =\mathbb{H}_r^\times$ for distinct nonzero $x$ and $y$. 
 
\medskip

 \noindent\textbf{Case 2:} {$r$ is even.} 

\smallskip

In this case, we claim that $i=0$ if $q \equiv 1$ ($\textrm{mod} \ 2r$) and $i=\frac{r}{2}$ if $q \equiv r+1$ ($\textrm{mod} \ 2r$), i.e., $0 \in [1]\hyperplus[1]$ if $q \equiv 1$ ($\textrm{mod} \ 2r$), and $0 \in [1]\hyperplus[g^{r/2}]$ if $q \equiv r+1$ ($\textrm{mod} \ 2r$). 

 \medskip\noindent\textbf{Case 2a:} $q \equiv 1$ ($\textrm{mod} \ 2r$).
 
Write $q= 2lr+1$. Then 
\[
  \frac{q-1}{2} \equiv lr \equiv 0.
\]

Thus $i=0$ and $\H$ is isomorphic to the hyperfield $\mathbb{H}_r$ defined above.

 \medskip\noindent\textbf{Case 2b:} $q \equiv r+1$ ($\textrm{mod} \ 2r$).

Write $q= (2l+1)r+1$. Then
\[
  \frac{q-1}{2} \equiv lr + \frac{r}{2} \equiv \frac{r}{2}.
\]

Thus $i=\frac{r}{2}$ and $\H$ is isomorphic to the hyperfield $\mathbb{H}'_r$ whose multiplicative group is cyclic of order $r$ (with a unique element $g'$ of order $2$) and whose hyperaddition is characterized by the identities $-x = g' x$ for all $x \in \mathbb{H}'_r$; $x \hyperplus g' x = \mathbb{H}'_r$ for $x \ne 0$; and $x \hyperplus y =(\mathbb{H}'_r)^\times$ for nonzero $x,y \in \mathbb{H}'_r$ with $y \neq g' x$.

\medskip
 
Finally, note that $\H_r$ and $\mathbb{H}'_r$ are not isomorphic, since if $$\phi: \mathbb{H}_r \to \mathbb{H}'_r$$ was an isomorphism of hyperfields we would have
$$\phi(\mathbb{H}_r) = \phi([1] \hyperplus [1]) = \phi([1]) \hyperplus' \phi([1]) = [1] \hyperplus' [1]=(\mathbb{H}'_r)^\times, $$
a contradiction.
 
This concludes the proof of Theorem~\ref{first main theorem}. 
\end{proof}


\section{Proof of Theorem~\ref{second main theorem}}
\label{section: proof of the second main theorem}

In this section we prove Theorem~\ref{second main theorem} concerning the structure of hyperfields of order at most $4$.
In addition to Theorem~\ref{first main theorem}, we will use the following result proved independently by Bergelson and Shapiro \cite[Theorem 1.3]{BergelsonShapiro} and Turnwald  \cite[Theorem 1]{Turnwald}:

\begin{theorem}[Bergelson--Shapiro, Turnwald] 
\label{thm:Turnwald}
If $F$ is an infinite field and $G$ is a subgroup of $F^\times$ of finite index then $G - G = F$.
\end{theorem}

Here, as usual, $G-G$ denotes the set $\{ x - y \; : \; x,y \in G \}$. 

For the reader's convenience, we present Turnwald's elegant short argument below; it is an application of the Hales-Jewett theorem \cite{HalesJewett} from Ramsey theory. (The argument of Bergelson--Shapiro uses a simpler variant of Ramsey's theorem plus the amenability of finite groups.)

In order to state the Hales-Jewett theorem, define a subset $L$ of $\{ 0,\ldots,m \}^N$ to be a {\em combinatorial line} if there is a partition of $\{ 1,\ldots,N \}$ into disjoint subsets $J_0$ and $J_1$, with $J_1 \neq \emptyset$, and elements $k_j' \in \{ 0,\ldots,m \}$ for $j \in J_0$ such that 
\[
L = \{ (k_1,\ldots,k_N) \; : \; k_j = k_j' {\rm \; for \; } j \in J_0 {\rm \; and \; } k_{j_1} = k_{j_2} {\rm \; for \; } j_1,j_2 \in J_1 \}.
\]
The Hales-Jewett theorem asserts that for every $m,n \in \N$ there exists $N(m,n) \in \N$ such that if $S$ is a set, $N \geq N(m,n)$, and $f : \{ 0, \ldots, m \}^N \to S$ is a function taking on at most $n$ values, then $f$ must be constant on some combinatorial line.

\begin{proof}[Proof of Theorem~\ref{thm:Turnwald}]
Let $x_0 = 0$ and let $x_1,\ldots,x_r \in F$ be a set of (left) coset representatives for $G$ in $F^\times$.
We claim that there exists $c \in F^\times$ such that $1 + cx_i \in G$ for all $i=1,\ldots,r$.
Given the claim, if we let $y_i = cx_i$ then $y_1,\ldots,y_r$ also form a set of coset representatives for $G$ and $1+y_i \in G$ for all $i$.
It follows that $y_i G \subseteq G - G$ for all $i$, and hence $G-G=F$ as desired.

To prove the claim, let $N=N(r,r+1)$ be the bound given in the Hales-Jewett theorem. Since $F$ is infinite, there exist $c_1,\ldots,c_N \in F$ so that $\sum_{j \in J} c_j \neq 0$ for all non-empty subsets $J \subseteq \{ 1,\ldots, N \}$ (choose $c_k$ inductively so that $\sum_{j \in J} c_j \neq 0$ for all $J \subseteq \{ 1,\ldots,k \}$). Let $\H = F / G$, and define $f : \{ 0,\ldots,r \}^N \to \H$ by $f(k_1,\ldots,k_N) = [\sum_{j=1}^N c_j x_{k_j}]$. By the Hales-Jewett theorem, there is a
partition of $\{ 1,\ldots,N \}$ into disjoint subsets $J_0$ and $J_1$, with $J_1 \neq \emptyset$, and elements $k_j' \in \{ 0,\ldots,r \}$ for $j \in J_0$ such that 
$f$ is constant on the corresponding combinatorial line $L$. Unwinding the definitions, this means that $[a + bx_k]$ is constant for all $0 \leq k \leq r$, where $a=\sum_{j \in J_0} c_j x_{k'_j}$ and $b = \sum_{j \in J_1} c_j$. Since $x_0 = 0$, this constant value is equal to $[a]$, and thus $aG = (a+bx_k)G$ for all $k=1\ldots,r$. Setting $c = a^{-1}b$ establishes the claim. 
(Note that $a \neq 0$ since $b \neq 0$ and $x_k \neq 0$ for $k \in \{ 1,\ldots,r \}$.)
\end{proof}

\begin{remark}
The same proof applies verbatim if $F$ is replaced by an infinite division ring which is not necessarily commutative.
\end{remark}

\begin{remark} \label{rem:odd index}
If $[F^\times:G]$ is odd, note that $-1 \in G$ since $[-1]^2 = [1]$ in the odd-order group $F^\times/G$, and hence $G-G = G+G$.
\end{remark}

\medskip

We are now ready to prove Theorem~\ref{second main theorem}.

\medskip

\begin{proof}[Proof of Theorem~\ref{second main theorem}] 

We consider separately the cases where $|\H |$ is equal to 2, 3, or 4.

 \medskip\noindent\textbf{Case 1:} Hyperfields of order $2$.
It is clear by inspection that every hyperfield of order 2 is isomorphic to either the Krasner hyperfield $\K$ or the field $\F_2$ of two elements.
Moreover, $\K$ is isomorphic to $F / F^\times$ for any field $F$ with $|F|>2$. This proves part (1) of Theorem~\ref{second main theorem}.

 \medskip\noindent\textbf{Case 2:} Hyperfields of order $3$.
 
Let $\H$ be a hyperfield with $| \H | = 3$. The multiplicative group $\H^\times$ must be cyclic of order $2$, so we can write $\H = \{ 0, 1, g \}$ with $g^2 = 1$ and $g \neq 1$. (We note that $g$ might or might not be equal to $-1$ in $\H$.) A straightforward case analysis using the hyperfield axioms now shows that there are precisely 5 possible hyperaddition structures for $\H$. Indeed, by the distributive law (Definition~\ref{def: hyperfield}(4)) the hyperaddition table for $\H$ is completely determined by $1 \hyperplus 1$ and $1 \hyperplus g$, and only the following 5 possibilities are compatible with the associativity of $\hyperplus$ and the reversibility axiom (Definition~\ref{def: hypergroup}(3)):

\begin{enumerate}
\item $1\hyperplus 1 = \{ g \}$ and $1 \hyperplus g = \{ 0 \}$.  In this case, $\H \cong \F_3$.
\item $1\hyperplus 1 = \{ 0,g \}$ and $1 \hyperplus g = \{ 1,g \}$.  In this case, $\H \cong \F_5 / G^2_5$.
\item $1\hyperplus 1 = \H$ and $1 \hyperplus g = \H^\times$.  In this case, $\H \cong \H_2 \cong \F_q / G^2_q$ for all prime powers $q \geq 7$ and $q \equiv 1$ ($\textrm{mod} \ 4$).
\item $1\hyperplus 1 = \H^\times$ and $1 \hyperplus g = \H$.  In this case, $\H \cong \H'_2 \cong \F_q / G^2_q$ for all prime powers $q \geq 7$ and $q \equiv 3$ ($\textrm{mod} \ 4$).
\item $1\hyperplus 1 = \{ 1 \}$ and $1 \hyperplus g = \H$.  In this case, $\H \cong \S \cong \R / \R_{>0}$.
\end{enumerate}

It is straightforward to verify that none of these five hyperfields is isomorphic to any other one. 

Moreover, $\S$ is not isomorphic to a quotient of any finite field $\F_q$ because otherwise there would be a homomorphism $\F_q \to \S$ and hence, by 
\cite[Section 3]{Marshall}, an ordering on $\F_q$; however, it is well known that every ordered field has characteristic zero.
(Indeed, if $F$ were an ordered field of characteristic $p>0$ then we would have $-1 = p-1 = 1 + 1 + \cdots + 1 \; ( p-1 \; {\rm times}) > 0$, a contradiction.)
Alternatively, using Remark~\ref{rem:betterNr}, it suffices to check that $\F_q / G^2_q \not\cong \S$ for $q  \leq 5$, which is straightforward.

This proves part (2) of Theorem~\ref{second main theorem}.

 \medskip\noindent\textbf{Case 3:} Hyperfields of order $4$.
 
Let $\H$ be a hyperfield with $| \H | = 4$. The multiplicative group $\H^\times$ must be cyclic of order $3$, so we can write $\H = \{ 0, 1, g, g^2 \}$ with $g^3 = 1$. 
Note in this case that $1 \neq -1$ in $\H$, since $\H^\times$ contains no element of order $2$.
A tedious, but still straightforward, case analysis using the hyperfield axioms now shows that there are precisely 9 possible hyperaddition structures for $\H$: 

\begin{enumerate}
\item $1\hyperplus 1 = \{ 0 \}$ and $1 \hyperplus g = \{ g^2 \}$. In this case, $\H \cong \F_4$. 
\item $1\hyperplus 1 = \{ 0, g \}$, and $1 \hyperplus g = \{ 1, g^2 \}$. In this case, $\H \cong \F_7 / G^3_7$. 
\item $1\hyperplus 1 = \{ 0, g^2 \}$ and $1 \hyperplus g = \{ g, g^2 \}$. In this case, $\H$ is isomorphic to the hyperfield in (2). 
\item $1\hyperplus 1 = \{ 0, g, g^2 \}$ and $1 \hyperplus g = \H^\times$. In this case, $\H \cong \F_{13} / G^3_{13} \cong \F_{16} / G^3_{16}$. 
\item $1\hyperplus 1 = \H $ and $1 \hyperplus g = \H^\times$. In this case, $\H \cong \F_q / G^3_q$ for all prime powers $q \geq 19$. 
\item $1\hyperplus 1 = \{ 0, g, g^2 \}$ and $1 \hyperplus g = \{ 1, g \}$. 
\item $1\hyperplus 1 = \{ 0, 1, g \}$ and $1 \hyperplus g = \{ 1, g^2 \}$. 
\item $1\hyperplus 1 = \{ 0, 1, g^2 \}$, and $1 \hyperplus g = \{ g, g^2 \}$. In this case, $\H$ is isomorphic to the hyperfield in (7). 
\item $1\hyperplus 1 = \H$ and $1 \hyperplus g = \{ 1, g \}$. 

\end{enumerate}

It is straightforward to verify that none of the hyperfields (1),(2),(4),(5),(6),(7),(9) is isomorphic to any other one. 

To see that the none of the hyperfields in (6),(7),(9) is a quotient of a finite field $\F_q$, it suffices by Remark~\ref{rem:betterNr} to check that none of these hyperfields is isomorphic to $\F_q / G^3_q$ for some $q  \leq 16$. We omit the details of this straightforward computation.

The hyperfield $\H$ in (9) can, however, be realized as the quotient of an infinite field. For example (cf.~\cite[Proposition 6]{LeepShapiro}), choose a prime number $p$, let ${\rm ord}_p$ denote the $p$-adic valuation on $\Q$, and let $G$ be the index 3 subgroup of $\Q^\times$ consisting of all rational numbers $x = a/b$ such that $\ord_p(x)$ is a multiple of $3$. It is straightforward to check, using the ultrametric inequality, that $\H \cong \Q / G$.

On the other hand, by Theorem~\ref{thm:Turnwald} and Remark~\ref{rem:odd index}, neither of the hyperfields (6),(7) can be the quotient of an infinite field since 
$G+G = G-G = F$ implies that $1 \hyperplus 1 = \H$, but in each of $(6),(7)$ we have $1 \hyperplus 1 \neq \H$.

This proves part (3) of Theorem~\ref{second main theorem}.
\end{proof}

\begin{remark}
Theorem~\ref{thm:Turnwald} also yields the following generalization of Massouros's example \cite{Massouros} mentioned in Section~\ref{intro}: if $\H$ is a hyperfield such that $\H^\times$ is not cyclic and $1\hyperplus (-1) \neq \H$, then $\H$ is not a quotient hyperfield.
Indeed, the former condition implies that $\H$ is not a quotient of a finite field and the latter implies that $\H$ is not a quotient of an infinite field.
\end{remark}


\section{Some open questions}
\label{section:open questions}

We conclude with some open questions.

\begin{enumerate}
\item What is the complete list of hyperfields of order 5, and which of these are quotients of finite (resp. infinite) fields? In principle the methods of this paper give an algorithm to enumerate all hyperfields of order 5 and determine which ones are quotients of finite fields, but the computations are probably too involved to do by hand and we have not attempted to write the necessary computer code. It is not clear {\em a priori} whether the methods of this paper will suffice to determine whether or not each such hyperfield is a quotient of some infinite field; some new ideas may be required.
\item Following up on question (1), is there a (practical, or even just theoretical) algorithm to determine whether or not a given finite hyperfield is a quotient of some infinite field? 
\item What is the true order of growth of $N_r$ in Theorem~\ref{first main theorem}? In other words, how small can we take $N_r$ so that all hyperfields of the form $\F_q / G^r_q$ with $q \geq N_r$ and $q \equiv 1 \pmod{r}$ are isomorphic to either $\H_r$ or $\H'_r$? 
For $r=2$ and $r=3$ the upper bound for $N_r$ given by Remark~\ref{rem:betterNr} is sharp; what happens in general? 
(Taking $\H = \F_{p^2} / \F_p^\times$ with $p=r-1$ shows that $N_r > (r-1)^2$ whenever $r-1$ is prime, but this is quite far from the upper bound furnished by Remark~\ref{rem:betterNr}.)
\item Let $H_r$ be the number of isomorphism classes of hyperfields of order $r+1$, and let $Q_r$ be the number of isomorphism classes of {\em quotient hyperfields} of order $r+1$.
How do the relative growth rates of $H_r$ and $Q_r$ compare? We conjecture, based on the considerations in this paper, that $Q_r / H_r$ tends to zero as $r$ tends to infinity. (If true, this would mean colloquially speaking that ``almost all'' hyperfields are not quotients of fields.) 

\end{enumerate}


.

\end{document}